\documentclass[reqno]{amsart}
\usepackage{amssymb,amsmath,amsthm,amscd,latexsym,amsfonts}
\usepackage{mathtools}
\usepackage[english]{babel}
\usepackage[T1]{fontenc}
\usepackage{amsaddr}
\usepackage{cite}

\DeclareMathOperator{\spann}{span}

\DeclareMathOperator{\Ann}{Ann}
\DeclareMathOperator{\Hom}{Hom}
\DeclareMathOperator{\GL}{GL}
\DeclareMathOperator{\Orb}{Orb}
\DeclareMathOperator{\gr}{gr}

\newtheorem{pr}{Proposition}
\newtheorem{cor}{Corollary}

\newtheorem{de}{Definition}
\newtheorem{teo}{Theorem}

\usepackage[height=22cm,top=2.5cm,left=2cm,right=2cm]{geometry}

\begin{document}
\title[Solvable Leibniz algebra with nilradical $F_n^1$ and its rigidity]{Solvable Leibniz algebra with non-Lie and non-split naturally graded filiform nilradical and its rigidity}

\author{M.~Ladra\textsuperscript{1}, K.~K.~Masutova\textsuperscript{1},  B.~A.~Omirov\textsuperscript{2}}

\address{\textsuperscript{1}Department of Algebra, University of Santiago de Compostela, 15782 Santiago de Compostela, Spain, manuel.ladra@usc.es, kamilyam81@mail.ru}
\address{\textsuperscript{2}Institute of Mathematics, National University of Uzbekistan, 100125, Tashkent,  Uzbekistan, omirovb@mail.ru}

\begin{abstract}

The description of complex solvable Leibniz algebras whose nilradical is a naturally graded filiform algebra is already known.
 Unfortunately, a mistake was made in that description. Namely, in the case where the dimension of the solvable Leibniz algebra with nilradical $F_n^1$ is equal to $n+2$,
  it was asserted that there is no such algebra. However, it was possible for us to find a unique $(n+2)$-dimensional solvable Leibniz algebra with nilradical $F_n^1$.
   In addition, we establish the triviality of the second group of cohomology for this algebra with coefficients in itself, which implies its rigidity.

\end{abstract}


\maketitle

\section{Introduction}

Recall that Leibniz algebras are a generalization of Lie algebras \cite{LoPi,Lod}.
 Active investigations on Leibniz algebra theory show that many results of the theory of Lie algebras can be extended to Leibniz algebras.
 From the theory of Lie algebras it is well known that the study of finite-dimensional Lie algebras can be reduced to
 the solvable \cite{Jac} and, due to the work of Malcev \cite{Mal}, to nilpotent ones.  In the Leibniz algebra case we have an analogue of Levi's
theorem, due to the result of  Barnes \cite{Bar}, that the decomposition of a Leibniz algebra into a semidirect sum of its solvable
radical and a semisimple Lie algebra. Therefore, the study of finite-dimensional Leibniz algebras can be reduced to solvable ones.

The methods of Mubarakzjanov \cite{Mub} allow to describe solvable Lie algebras with a given nilradical.
Applications of these methods were used in several articles which are devoted to the study of solvable Lie algebras with various types of nilradical,
 such as naturally graded filiform and quasi-filiform algebras, abelian, triangular, etc.
\cite{AnCaGa1,AnCaGa2,NdWi,TrWi,WaLiDe}. The extension of these methods to the Leibniz algebras case was proposed in \cite{CLOK2}.

Let $V$ be an $n$-dimensional vector space over the field of complex numbers. The bilinear maps $V\times V \to V$ form an $n^3$-dimensional affine space $B(V)$ over $\mathbb{C}$.
It is known that an $n$-dimensional Leibniz algebra may be considered as an element $\lambda$ of the affine variety
$\Hom(V\otimes V,V)$ via the bilinear mapping $\lambda \colon V\otimes V\to V$, defining the Leibniz bracket on the vector space $V$.
 Since Leibniz algebras are defined via polynomial identities, the set of Leibniz algebra structures forms an algebraic subset of the variety $\Hom(V\otimes V,V)$
  and the linear reductive group $\GL_n(F)$ acts on the set as follows:
\[(g*\lambda)(x,y)=g(\lambda(g^{-1}(x),g^{-1}(y))).\]
The orbits ($\Orb(-)$) under this action are the  isomorphism classes of algebras.

 Recall that a set is called irreducible if it cannot be represented as a union of two non-trivial closed subsets (with respect to the Zariski topology in  $B(V)$),
otherwise it is called reducible. The maximal irreducible closed subset of a variety is called an irreducible component.
 From algebraic geometry we know that an algebraic variety is a union of irreducible components and that closures of open sets give rise to irreducible components.
  Therefore, for the description of a variety it is very important to find all open sets.
 Since under the above action the variety of Leibniz algebras consists of orbits of algebras, the description of the variety is reduced to  finding the open orbits.
 In any variety of algebras there are algebras with open orbits (so-called rigid algebras). Thus, the closures of orbits of rigid algebras give irreducible components of the variety.
 Hence, in order to describe the variety of algebras it is enough to describe all rigid algebras.
 Due to the work \cite{Bal}, we can apply the general principles for deformations and rigidity of Leibniz algebras. In fact, it is proved that the
nullity of the second cohomology group with coefficients in itself gives a sufficient condition for the rigidity of the algebra.

In this paper we fix a mistake which was made in the description of complex $(n+2)$-dimensional solvable Leibniz algebras whose nilradical
 is a non-Lie and non-split naturally graded filiform algebra (see \cite{CLOK1}). Namely, in that description \cite{CLOK1}, the authors assert that there is no such solvable Leibniz algebra.
  However, we find in our ameliorated description that  there exits a unique $(n+2)$-dimensional solvable algebra Leibniz with nilradical $F_n^1$.
   Moreover, we establish that the second group of cohomology of this Leibniz algebra is trivial, and consequently it is a rigid algebra.

  We will briefly describe the organization of the paper. In Section~\ref{sec:preliminaries} we recall  definitions and known results.
In Section~\ref{sec:main}  we find a unique $(n+2)$-dimensional solvable algebra Leibniz $R(F_n^1)$ with nilradical $F_n^1$.
   Moreover, we establish the triviality of the second group of cohomology of $R(F_n^1)$ with coefficients in itself, and  therefore its rigidity.

\section{Preliminaries} \label{sec:preliminaries}

In this section we give necessary definitions and preliminary results.

\begin{de}[\cite{Lod}] A Leibniz algebra over a field $\mathbb{F}$ is a vector space $L$
equipped with a bilinear map, called bracket,
\[[-,-] \colon  L \times  L \rightarrow  L,\]
satisfying the Leibniz identity
\[ \big[x,[y,z]\big]=\big[[x,y],z\big]-\big[[x,z],y\big], \]
for all $x,y,z \in  L$.
\end{de}

The set $\Ann_r(L)=\{x \in L \mid [y,x]=0, \  y \in L\}$ is
called \emph{the right annihilator of the Leibniz algebra $L$}. Note
that $\Ann_r(L)$ is an ideal of $L$ and for any $x, y \in L$ the
elements $[x,x]$, $[x,y]+ [y,x]\in \Ann_r(L)$.

\begin{de} A linear map $d \colon L \rightarrow L$ of a Leibniz algebra $(L, [-, -])$ is said to be a derivation if for all $x, y \in L$ the following condition holds:
\[d([x,y])=[d(x),y] + [x, d(y)].\]
\end{de}

\subsection{Solvable Leibniz algebras.}

For a Leibniz algebra $L$ we consider the following \emph{lower central} and
\emph{derived series}:
\begin{align*}
 L^1& =L  \qquad \ \ L^{k+1} =[L^k,L^1], \qquad  \ \ \ k \geq 1; \\
 L^{[1]} & = L, \qquad L^{[s+1]} = [L^{[s]}, L^{[s]}], \qquad s \geq 1.
\end{align*}

\begin{de} A Leibniz algebra $L$ is said to be nilpotent (respectively,
solvable), if there exists $n\in\mathbb N$ ($m\in\mathbb N$) such
that $L^{n}=0$ (respectively, $L^{[m]}=0$).

\end{de}

Obviously, the index of nilpotency of an $n$-dimensional nilpotent
Leibniz algebra is not greater than $n+1$.

It should be noted that the sum of any two nilpotent (solvable) ideals is
nilpotent (solvable).

\begin{de} The  maximal nilpotent (solvable) ideal of a Leibniz algebra is said to be a nilradical (solvable radical) of the algebra.
\end{de}

\begin{de} A Leibniz algebra $L$ is said to be filiform if $\dim L^i=n-i$, where $n=\dim L$ and $2\leq i \leq n$.
\end{de}

Now let us define a natural graduation for a filiform Leibniz algebra.

\begin{de} Given a filiform Leibniz algebra $L$, put $L_i=L^i/L^{i+1}, \ 1 \leq i\leq n-1$,  and $\gr(L) = L_1 \oplus
L_2\oplus\dots \oplus L_{n-1}$. Then $[L_i,L_j]\subseteq L_{i+j}$ and we obtain the graded algebra $\gr(L)$.
 If $\gr(L)$ and $L$ are isomorphic, then we say that the algebra $L$ is naturally graded.
\end{de}

Due to  \cite{AyOm2} and \cite{Ver} it is known that there
are three naturally graded filiform Leibniz algebras. In
fact, the third type encloses the class of naturally graded
filiform Lie algebras.

\begin{teo} Any complex $n$-dimensional naturally graded filiform
Leibniz algebra is isomorphic to one of the following pairwise non
isomorphic algebras:
\begin{align*}
F_n^1: &  \  [e_i,e_1]=e_{i+1}, \qquad  \qquad  \qquad    \qquad  \qquad  \qquad \quad \ \  2\leq i \leq {n-1},\\
F_n^2: & \  [e_i,e_1]=e_{i+1}, \qquad  \qquad  \qquad  \qquad  \qquad   \qquad \quad \ \ 1\leq i \leq {n-2},\\
F_n^3(\alpha): &  \  \begin{cases}
[e_i,e_1]=-[e_1,e_i]=e_{i+1}, &
2\leq i \leq {n-1},\\
[e_i,e_{n+1-i}]=-[e_{n+1-i},e_i]=\alpha (-1)^{i+1}e_n, & 2\leq i\leq n-1,
\end{cases}
\end{align*}
where $\alpha\in\{0,1\}$ for even $n$ and $\alpha=0$ for odd $n$.
\end{teo}

The following theorems describe solvable Leibniz algebras with nilradical $F_n^1$.

\begin{teo}[\cite{CLOK1}] \label{thR_2} An arbitrary $(n+1)$-dimensional solvable Leibniz algebra with nilradical $F_n^1$ is isomorphic to one of the following pairwise non-isomorphic algebras:
\[ R_1, \quad  R_2(\alpha), \quad  R_3, \quad  R_4, \quad R_5(\alpha_4,\dots,\alpha_{n-1}), \quad R_6(\alpha_4,\dots,\alpha_{n-1}), \quad R_7(\alpha_4,\dots,\alpha_{n-1}) .\]
\end{teo}

\begin{teo}[\cite{CLOK1}] \label{thR_3} There does not exist any $(n+2)$-dimensional solvable Leibniz algebra with nilradical $F_n^1$.
\end{teo}

Nevertheless, the last theorem is not true. We will find only one $(n+2)$-dimensional solvable algebra Leibniz with nilradical $F_n^1$.

\subsection{The second cohomology group of a Leibniz algebra.}

For acquaintance with the definition of cohomology group of Leibniz algebras and its applications
to the description of the variety of Leibniz algebras (similar to the Lie algebras case) we refer the reader
to the papers \cite{Bal,Ger,GoKh,Lod,LoPi,NiRi}.
Here we just recall that the second cohomology group of a Leibniz algebra $L$ with coefficients in a corepresentation $M$
is the quotient space
\[HL^2(L, M) = ZL^2(L, M)/BL^2(L, M),\]
where the 2-cocycles $\varphi\in ZL^2(L, M)$ and the 2-coboundaries $f\in BL^2(L, M)$
are defined as follows

\begin{equation}\label{E.Z2}
(d^2\varphi)(a,b,c)=[a,\varphi(b,c)]-[\varphi(a,b),c]+[\varphi(a,c),b]+\varphi(a,[b,c])-
\varphi([a,b],c)+\varphi([a,c],b)=0
\end{equation}
and
\[ f(a,b)=[d(a),b]+[a,d(b)]-d([a,b]) \ \ \text{for some linear map} \ d. \]

\section{(n+2)-dimensional solvable Leibniz algebra with naturally graded \\ filiform  nilradical $F_n^1$ and its rigidity} \label{sec:main}

\subsection{Solvable Leibniz algebras with naturally graded filiform nilradical $F_n^1$}

As mentioned above, in the work \cite{CLOK1} it was made a mistake. Therefore, we consider a solvable Leibniz algebra whose dimension is equal to $n+2$
 and the nilradical is the filiform Leibniz algebra $F_n^1$. We give a complete proof of the following theorem and specify the exact place of the error in Theorem \ref{thR_3}.

\begin{teo} An arbitrary $(n+2)$-dimensional solvable Leibniz algebra with nilradical $F_n^1$ is isomorphic to the algebra $R(F_n^1)$ with the multiplication table:
\[\begin{array}{llll}
[e_i,e_1]=e_{i+1}, &  2\leq i \leq {n-1},  & [e_1,x]=e_1,      &                  \\[1mm]
[e_i,y]=e_i,       &  2\leq i\leq n,       & [e_i,x]=(i-1)e_i, &  2\leq i \leq n, \\[1mm]
                   &                       & [x,e_1]=-e_1.     &                  \\[1mm]
\end{array}\]
\end{teo}
\begin{proof} From the conditions of the theorem we have the existence of a basis $\{e_1, e_2, \dots, e_n, x, y\}$ such that the multiplication  table of $F_n^1$ remains the same.
 The outer non-nilpotent derivations of $F_n^1$, denoted by $D_x$ and $D_y$, are of the form given in  \cite[Proposition 4.1]{CLOK1},
  with the set of entries $\{\alpha_i, \gamma\}$ and $\{\beta_i,\delta\}$, respectively, where $[e_i,x]=D_x(e_i)$ and $[e_i,y]=D_y(e_i)$.

By taking the following change of basis:
\[x'=\frac{\beta_2}{\alpha_1\beta_2-\alpha_2\beta_1}x-
\frac{\alpha_2}{\alpha_1\beta_2-\alpha_2\beta_1}y,\quad \quad y'=-\frac{\beta_1}{\alpha_1\beta_2-\alpha_2\beta_1}x+
\frac{\alpha_1}{\alpha_1\beta_2-\alpha_2\beta_1}y,\]
we may assume that $\alpha_1=\beta_2=1$ and $\alpha_2=\beta_1=0$.

Therefore, we have the products
\[\begin{array}{ll}
[e_1,x]=e_1+\sum\limits_{i=3}^n\alpha_ie_i,                & [e_2,x]=e_2+\sum\limits_{i=3}^{n-1}\alpha_ie_i+\gamma e_n, \\[1mm]
[e_i,x]=(i-1)e_i+\sum\limits_{j=i+1}^n\alpha_{j-i+2}e_j,   & 3\leq i\leq n,\\[1mm]
[e_1,y]=e_2+\sum\limits_{i=3}^n\beta_ie_i,                 & [e_2,y]=e_2+\sum\limits_{i=3}^{n-1}\beta_ie_i+\delta e_n,\\[1mm]
[e_i,y]=e_i+\sum\limits_{j=i+1}^n\beta_{j-i+2}e_j,         & 3\leq i\leq n.\\[1mm]
\end{array}\]

Let us introduce the  notations

\[[x,e_1]=\sum\limits_{i=1}^n\lambda_ie_i, \qquad [x,e_2]=\sum\limits_{i=1}^n\delta_ie_i, \qquad [x,x]=\sum\limits_{i=1}^n\mu_ie_i.\]

From the  Leibniz identity we get $[x,e_i]=0, \ 3\leq i\leq n$.

By applying similar arguments as in Case 1 of the proof of Theorem \ref{thR_2} (see \cite[Theorem 4.2]{CLOK1}) and taking into account that the products $[e_i, y], \ 1\leq i\leq n$,
 will not be changed under the transformations of the bases that were used there, we obtain two cases of products of elements:

\[\text{Case} \ \lambda_2\neq 1:
\left\{\begin{array}{ll}
[e_i,e_1]=e_{i+1},  & \  2\leq i \leq {n-1},\\[1mm]
[e_1,x]=e_1,        &  \\[1mm]
[e_i,x]=(i-1)e_i,   & \  2\leq i \leq n,\\[1mm]
[x,e_1]=-e_1-e_2,   &  \\[1mm]
[e_1,y]=\beta e_n, &  \\[1mm]
[e_i,y]=e_i+\sum\limits_{j=i+1}^n\beta_{j-i+2}e_j,         & 2\leq i\leq n;
\end{array}\right. \]

\[\text{Case} \ \lambda_2= 1:
\left\{\begin{array}{ll}
[e_i,e_1]=e_{i+1},  & \  2\leq i \leq {n-1},\\[1mm]
[e_1,x]=e_1,        &  \\[1mm]
[e_i,x]=(i-1)e_i,   & \  2\leq i \leq n,\\[1mm]
[x,e_1]=-e_1,       &  \\[1mm]
[e_1,y]=\delta e_n, &  \\[1mm]
[e_i,y]=e_i+\sum\limits_{j=i+1}^n\beta_{j-i+2}e_j,         & 2\leq i\leq n.
\end{array}\right. \]

That omission of Case $\lambda_2=1$ in the proof of Theorem \ref{thR_3} is an error in the work \cite{CLOK1}.
We fill this gap in the proof.

Let us introduce notations

\[[y,e_1]=\sum\limits_{i=1}^n\eta_ie_i, \quad [y,e_2]=\sum\limits_{i=1}^n\theta_ie_i, \quad [y,y]=\sum\limits_{i=1}^n\tau_ie_i, \quad [x,y]=\sum\limits_{i=1}^n\sigma_ie_i, \quad [y,x]=\sum\limits_{i=1}^n\rho_ie_i.\]

Consider the \textbf{Case $\lambda_2\neq 1$.} From the Leibniz identity
\[[x,[e_1,y]]=[[x,e_1],y]-[[x,y],e_1], \] we get
\[0=-e_2-\sum\limits_{i=3}^n\beta_ie_i-\sum\limits_{i=3}^{n}\sigma_{i-1}e_i-\beta e_n.\]

Thus, we have a contradiction with the assumption of the existence of an algebra under the conditions of the theorem.

Now we consider the \textbf{Case $\lambda_2= 1$.}

By taking the following change of basis elements $y'=y+\sigma_1 e_1$, we derive $[x,y]=\sum\limits_{i=2}^n\sigma_ie_i$.

From the Leibniz identity $0=[e_2,[x,y]]=[[e_2,x],y]-[[e_2,y],x]$, we deduce $\beta_i=0, \ 3\leq i\leq n$.

By making the change $y'=y+\sum\limits_{i=2}^{n-1}\eta_{i+1}e_i$, we obtain $[y, e_1]=\eta_1e_1+\eta_2e_2$ and
\[\begin{array}{lll}
[e_i,e_1]=e_{i+1}, \  2\leq i \leq {n-1},  & [e_1,y]=\delta e_n,                    & [y,e_i]=\sum\limits_{j=1}^n\mu_{i,j}e_j, \ 3\leq i\leq n,\\[1mm]
[e_1,x]=e_1,                               & [e_i,y]=e_i, \ 2\leq i\leq n,          & [y,x]=\sum\limits_{i=1}^n\rho_ie_i,\\[1mm]
[e_i,x]=(i-1)e_i,  \  2\leq i \leq n,      & [y,e_1]=\eta_1e_1+\eta_2e_2,           & [y,y]=\sum\limits_{i=1}^n\tau_ie_i,\\[1mm]
[x,e_1]=-e_1,                              & [y,e_2]=\sum\limits_{i=1}^n\theta_ie_i,& [x,y]=\sum\limits_{i=2}^n\sigma_ie_i. \\[1mm]
\end{array}\]

By applying the Leibniz identity to the products above we derive

\begin{align*}
  \delta & =\eta_1=\eta_2=\theta_i=\mu_{i,j}=\sigma_i=0 \qquad (1\leq i, j\leq n), \\
  \rho_i & =\tau_i=0 \quad (1\leq i\leq n-1), \qquad \rho_n=(n-1)\tau_n.
\end{align*}

Finally, by applying the change of basis elements $y'=y-\tau_ne_n$, we obtain the algebra $R(F_n^1)$.
\end{proof}

\subsection{Rigidity of the algebra $R(F_n^1)$.}

In order to describe the second group of cohomology of the algebra $R(F_n^1)$ we need the description of its derivations. The general matrix form of the derivations of $R(F_n^1)$ is given in the following proposition.

\begin{pr} \label{pr1} Any derivation of the algebra $R(F_n^1)$ has the following matrix form:
\[\left( \begin{array}{ccccccccc}
\alpha_1 & 0       & 0                & 0                 &\dots     & 0                     & 0                 & 0     &0\\
0        & \beta_2 & \beta_3          & 0                 &\dots     & 0                     & 0                 & 0     &0\\
0        & 0       & \alpha_1+\beta_2 & \beta_3           &\dots     & 0                     & 0                 & 0     &0\\
0        & 0       & 0                & 2\alpha_1+\beta_2 &\dots     & 0                     & 0                 & 0     &0\\
\vdots   &\vdots   & \vdots           & \vdots            &\ddots    &\vdots                 &\vdots        &\vdots &\vdots\\
0        & 0       & 0                & 0                 &\dots     & (n-3)\alpha_1+\beta_2 & \beta_3           & 0     &0\\
0        & 0       & 0                & 0                 &\dots     & 0                     & (n-2)\alpha_1+\beta_2 &0  &0\\
-\beta_3 & 0       & 0                & 0                 &\dots     & 0                     & 0                 & 0     &0\\
0        & 0       & 0                & 0                 &\dots     & 0                     & 0                 & 0     &0\\
\end{array}\right).\]
\end{pr}

\begin{proof}
The proof is carried out by straightforward calculations of the derivation property.
\end{proof}

From Proposition \ref{pr1} we conclude that $\dim  BL^2(R(F_n^1),R(F_n^1))=n^2+4n+1$. The general form of an element of the space $ZL^2(R(F_n^1),R(F_n^1))$ is presented below.

\begin{pr} \label{pr2}  An arbitrary element $\varphi$ of the space $ZL^2\big(R(F_n^1),R(F_n^1)\big)$ has the following form:
\[\left\{\begin{array}{ll}
\varphi(e_1,e_1) = \sum\limits_{k=3}^{n}a_{1,1}^ke_k,  &  \\[1mm]
\varphi(e_i,e_1) = \sum\limits_{k=1}^{n}a_{i,1}^ke_k+a_{i,1}^{n+1}x_1+a_{i,1}^{n+2}x_2,  & \ 2\leq i\leq n-1, \\[1mm]
\varphi(e_n,e_1) =a_{n-1,1}^{n+1}e_1-\sum\limits_{k=3}^{n-1}(\sum\limits_{j=1}^{k-2}a_{n-j,1}^{k-j})e_k+
\Big(\frac{n(n-1)}{2}a_{n+1,1}^{n+1}+(n-1)a_{n+1,1}^{n+2}-\sum\limits_{k=2}^{n-1}a_{k,1}^k\Big)e_n,  &  \\[1mm]
\varphi(e_1,e_2) =a_{1,2}^1e_1,  &  \\[1mm]
\varphi(e_i,e_2) = ((i-2)a_{1,2}^1+a_{2,2}^2)e_i+a_{2,2}^3e_{i+1}, & 2\leq i\leq n, \\[1mm]
\varphi(e_1,e_i) =-a_{i-1,1}^{n+1}e_1, & 3\leq i\leq n, \\[1mm]
\varphi(e_i,e_3) = -((i-1)a_{2,1}^{n+1}+a_{2,1}^{n+2})e_i-(a_{1,2}^1+a_{2,1}^1)e_{i+1}, & 2\leq i\leq n, \\[1mm]
\varphi(e_i,e_j) = -((i-1)a_{j-1,1}^{n+1}+a_{j-1,1}^{n+2})e_i+(a_{j-2,1}^{n+1}-a_{j-1,1}^1)e_{i+1},\ \ \ \ \ \ \ \ \ \ \ \ \ 4\leq j\leq n,& 2\leq i\leq n, \\[1mm]
\varphi(x_1,e_1)=a_{n+1,1}^1e_1+a_{1,1}^3e_2+\sum\limits_{k=3}^{n}a_{n+1,1}^ke_k+a_{n+1,1}^{n+1}x_1+a_{n+1,1}^{n+2}x_2,&\\[1mm]
\varphi(x_2,e_1)=a_{n+2,1}^1e_1+\sum\limits_{k=3}^{n}a_{n+2,1}^ke_k,&\\[1mm]
\varphi(x_1,e_2) =-a_{2,2}^3e_1,                                    & \\[1mm]
\varphi(x_1,e_3) =(a_{1,2}^1+a_{2,1}^1)e_1,                         & \\[1mm]
\varphi(x_1,e_i) =(a_{i-1,1}^1-a_{i-2,1}^{n+1})e_1,                 & 4\leq i\leq n, \\[1mm]
\varphi(e_1,x_1) =-a_{n+1,1}^1e_1+\sum\limits_{k=3}^{n-1}(k-2)a_{1,1}^{k+1}e_k+(n-2)a_{1,n+2}^ne_n
-a_{n+1,1}^{n+1}x_1-a_{n+1,1}^{n+2}x_2,  &  \\[1mm]
\varphi(e_2,x_1)=\sum\limits_{k=2}^na_{2,n+1}^ke_k-a_{1,2}^1x_1+(a_{1,2}^1-a_{2,2}^2)x_2,  &  \\[1mm]
\varphi(e_3,x_1)=(a_{1,2}^1+a_{2,1}^1)e_1+(a_{2,1}^2-a_{n+1,1}^{n+1}-a_{n+1,1}^{n+2})e_2+(a_{2,n+1}^2-a_{n+1,1}^1)e_3+&\\[1mm]
\ \ \ \ \ \ \ \ \ \ \ \ \ +\sum\limits_{k=4}^n(a_{2,n+1}^{k-1}-(k-3)a_{2,1}^k)e_k+2a_{2,1}^{n+1}x_1+2a_{2,1}^{n+2}x_2,&\\[1mm]
\varphi(e_i,x_1)=(i-2)(a_{i-1,1}^1-a_{i-2,1}^{n+1})e_1+\sum\limits_{k=2}^{i-2}(i-k)\sum\limits_{j=1}^{k-1}a_{i-j,1}^{k-j+1}e_k+&\\[1mm]
\ \ \ \ \ \ \ \ +\Big(\sum\limits_{k=2}^{i-1}a_{k,1}^k-\frac{(i-1)(i-2)}{2}a_{n+1,1}^{n+1}-(i-2)a_{n+1,1}^{n+2}\Big)e_{i-1}+\big(a_{2,n+1}^2-(i-2)a_{n+1,1}^1\big)e_i+&\\[1mm]
\ \ \ \ \ \ \ \ +\sum\limits_{k=i+1}^n\Big(a_{2,n+1}^{k-i+2}-(k-i)\sum\limits_{j=2}^{i-1}a_{j,1}^{k-i+j+1}\Big)e_k+(i-1)a_{i-1,1}^{n+1}x_1+(i-1)a_{i-1,1}^{n+2}x_2,  & 4\leq i\leq n, \\[1mm]
\varphi(e_1,x_2) =-a_{n+2,1}^1e_1+\sum\limits_{k=2}^{n-1}a_{1,1}^{k+1}e_k+a_{1,n+2}^ne_n, &  \\[1mm]
\varphi(e_2,x_2)=-a_{2,2}^3e_1+a_{2,n+2}^2e_2+a_{2,n+2}^3e_3-a_{1,2}^1x_1+(a_{1,2}^1-a_{2,2}^2)x_2,  &  \\[1mm]
\varphi(e_3,x_2)=(a_{1,2}^1+a_{2,1}^1)e_1+(a_{2,n+2}^2-a_{n+2,1}^1)e_3+a_{2,n+2}^3e_4+a_{2,1}^{n+1}x_1+a_{2,1}^{n+2}x_2,&\\[1mm]
\varphi(e_i,x_2)=(a_{i-1,1}^1-a_{i-2,1}^{n+1})e_1+(a_{2,n+2}^2-(i-2)a_{n+2,1}^1)e_i+a_{2,n+2}^3e_{i+1}+a_{i-1,1}^{n+1}x_1+a_{i-1,1}^{n+2}x_2,  &  4\leq i\leq n,\\[1mm]
\varphi(x_1,x_1)=\sum\limits_{k=2}^{n-2}(k-1)(a_{n+1,1}^{k+1}-a_{1,1}^{k+2})e_k+(n-2)(a_{n+1,1}^n-a_{1,n+2}^n)e_{n-1}+(n-1)a_{n+1,n+2}^ne_n,  &  \\[1mm]
\varphi(x_2,x_1)=a_{2,n+2}^3e_1+\sum\limits_{k=2}^{n-1}(k-1)a_{n+2,1}^{k+1}e_k+(n-1)a_{n+2,n+2}^ne_n,  &  \\[1mm]
\varphi(x_1,x_2)=-a_{2,n+2}^3e_1+\sum\limits_{k=2}^{n-2}(a_{n+1,1}^{k+1}-a_{1,1}^{k+2})e_k+(a_{n+1,1}^n-a_{1,n+2}^n)e_{n-1}+a_{n+1,n+2}^ne_n,  &  \\[1mm]
\varphi(x_2,x_2)=\sum\limits_{k=2}^{n-1}a_{n+2,1}^{k+1}e_k+a_{n+2,n+2}^ne_n. &
\end{array}\right.\]
\end{pr}

\begin{proof}
Let $\varphi\in ZL^2(R(F_n^1),R(F_n^1))$. We set
\begin{align*}
\varphi(e_i,e_j) & =\sum\limits_{k=1}^na_{i,j}^ke_k+a_{i,j}^{n+1}x_1+a_{i,j}^{n+2}x_2, && 1\leq i,j\leq n,\\
\varphi(e_i,x_j) & =\sum\limits_{k=1}^na_{i,n+j}^ke_k+a_{i,n+j}^{n+1}x_1+a_{i,n+j}^{n+2}x_2, && 1\leq i\leq n, \ 1\leq j\leq 2,\\
\varphi(x_j,e_i) & =\sum\limits_{k=1}^na_{n+j,i}^ke_k+a_{n+j,i}^{n+1}x_1+a_{n+j,i}^{n+2}x_2, && 1\leq i\leq n, \ 1\leq j\leq 2,\\
\varphi(x_i,x_j) & =\sum\limits_{k=1}^na_{n+i,n+j}^ke_k+a_{n+i,n+j}^{n+1}x_1+a_{n+i,n+j}^{n+2}x_2, &&  1\leq i,j\leq 2.
\end{align*}

For $\varphi\in ZL^2(R(F_n^1),R(F_n^1))$ we shall verify  equation \eqref{E.Z2}.

Consider $b=c$. Then we get $[a,\varphi(b,b)]=0$ for all $a\in L$. From the multiplication table of the algebra $R(F_n^1)$ it is easy to see that $\varphi(b,b)\in I \ \text{for all}\ b\in L$,
where $I=\spann \langle e_2,\dots,e_n\rangle$.

If $b,c\in I$, then we obtain $[a,\varphi(b,c)]=0$ for all $a\in L$, and consequently, $\varphi(b,c)\in I$, i.e.
$\varphi(I,I)\subseteq I$.

We consider the combinations of $(a,b,c), \ \text{with} \ a,b,c\in R(F_n^1)$.

If $(a,b,c)=(x_2,x_1,e_i)$, $i\neq 1$, then we derive $\varphi(x_2,I)\subseteq Q$, where $Q$  is the complementary vector
space of the nilradical $F_n^1$ to the algebra $R(F_n^1)$. The triple $(a,b,c)=(e_i,x_2,e_j)$, $j\neq 1$, implies $[e_i,\varphi(x_2,e_j)]=0$, from which we get $\varphi(x_2,I)=0$.

If we take $(a,b,c)=(e_i,x_1,e_j)$, $ j\neq 1$, then we deduce $\varphi(x_1,e_j)\in \langle e_1,x_2 \rangle$.

By considering $(a,b,c)=(x_1,e_i,e_1)$ and $(a,b,c)=(x_1,e_i,x_1)$, $ i\neq 1$, we obtain $\varphi(x_1,I)\subseteq  \langle e_1 \rangle$.

The triples $(a,b,c)=(e_1,x_1,e_i)$, $i\neq 1$, imply $\varphi(e_1,I)\subseteq \langle e_1,e_2 \rangle$.

By considering the triples $(a,b,c)=(e_1,x_2,e_i)$, $i\neq 1$, we get $\varphi(e_1,I)\subseteq \langle e_1 \rangle$.

Similarly, we obtain
\begin{align*}
(e_n,x_1,e_1) & \Rightarrow \ \varphi(e_n,e_1)\in F_n^1;\\
(x_1,x_2,e_1) & \Rightarrow \ \varphi(e_1,x_2)\in F_n^1;\\
(x_2,x_i,e_1), \, i=1,2 & \Rightarrow \ \varphi(x_2, e_1)\in F_n^1 \setminus \{e_2\};\\
(x_1,e_i,x_1), \,  i\neq 1 & \Rightarrow \ \varphi(e_2, x_1)\in R(F_n^1) \setminus \{e_1\};\\
(e_1,e_1,x_2) & \Rightarrow \ \varphi(e_1, e_1)\in I \setminus \{e_2\}.
\end{align*}

Given the restrictions above, by applying the multiplication of the algebra $R(F_n^1)$ and the property of cocycle for $(d^2\varphi)(e_1,e_i,e_1)=0$, we obtain
\[\varphi(e_1,e_2)=a_{1,2}^1e_1, \qquad \varphi(e_1,e_i)=-a_{i-1,1}^{n+1}e_1,\quad 3\leq i\leq n.\]

Analogously, from $(d^2\varphi)(e_i,e_1,e_j)=0$ and $(d^2\varphi)(e_i,e_j,e_1)=0, \ 1\leq i,j\leq n$, we conclude
\begin{align*}
\varphi(e_1,e_1)  & = \sum\limits_{k=3}^{n}a_{1,1}^ke_k, \\
\varphi(e_n,e_1)& =a_{n-1,1}^{n+1}e_1+\sum\limits_{k=2}^na_{n,1}^ke_k,\\
\varphi(e_i,e_2) & =((i-2)a_{1,2}^1+a_{2,2}^2)e_i+\sum\limits_{k=i+1}^na_{2,2}^{k-i+2}e_k, \qquad 2\leq i\leq n,\\
\varphi(e_i,e_3) & = -\big((i-1)a_{2,1}^{n+1}+a_{2,1}^{n+2}\big)e_i-(a_{1,2}^1+a_{2,1}^1)e_{i+1}, \quad 2\leq i\leq n,\\
\varphi(e_i,e_j) & = -\big((i-1)a_{j-1,1}^{n+1}+a_{j-1,1}^{n+2}\big)e_i+(a_{j-2,1}^{n+1}-a_{j-1,1}^1)e_{i+1}, \quad 2\leq i\leq n,\ 4\leq j\leq n.
\end{align*}

The equations $(d^2\varphi)(e_i,e_j,x_1)=0, \ 1\leq i,j\leq n$, imply
\begin{align*}
\varphi(e_1,x_1) & =a_{1,n+1}^1e_1+\sum\limits_{k=3}^{n-1}(k-2)a_{1,1}^{k+1}e_k+a_{1,n+1}^ne_n
+a_{1,n+1}^{n+1}x_1+a_{1,n+1}^{n+2}x_2,\\
\varphi(e_2,x_1) & =\sum\limits_{k=2}^na_{2,n+1}^ke_k-a_{1,2}^1x_1+(a_{1,2}^1-a_{2,2}^2)x_2,\\
\varphi(e_3,x_1) & =(a_{1,2}^1+a_{2,1}^1)e_1+(a_{2,1}^2+a_{1,n+1}^{n+1}+a_{1,n+1}^{n+2})e_2+(a_{2,n+1}^2+a_{1,n+1}^1)e_3\\
{} & +\sum\limits_{k=4}^n\big(a_{2,n+1}^{k-1}-(k-3)a_{2,1}^k\big)e_k+2a_{2,1}^{n+1}x_1+2a_{2,1}^{n+2}x_2,\\
\varphi(e_i,x_1) & =(i-2)(a_{i-1,1}^1-a_{i-2,1}^{n+1})e_1+\sum\limits_{k=2}^{i-2}(i-k)\sum\limits_{j=1}^{k-1}a_{i-j,1}^{k-j+1}e_k\\
{} & +\Big(\sum\limits_{k=2}^{i-1}a_{k,1}^k+\frac{(i-1)(i-2)}{2}a_{1,n+1}^{n+1}+(i-2)a_{1,n+1}^{n+2}\Big)e_{i-1}+(a_{2,n+1}^2+(i-2)a_{1,n+1}^1)e_i\\
{} & +\sum\limits_{k=i+1}^n\Big(a_{2,n+1}^{k-i+2}-(k-i)\sum\limits_{j=2}^{i-1}a_{j,1}^{k-i+j+1}\Big)e_k+(i-1)a_{i-1,1}^{n+1}x_1+(i-1)a_{i-1,1}^{n+2}x_2, \qquad 4\leq i\leq n,\\
\varphi(e_n,e_1) & =a_{n-1,1}^{n+1}e_1-\sum\limits_{k=3}^{n-1}\Big(\sum\limits_{j=2}^{k-2}a_{n-j,1}^{k-j}\Big)e_k-
\Big(\frac{n(n-1)}{2}a_{1,n+1}^{n+1}+(n-1)a_{1,n+1}^{n+2}+\sum\limits_{k=2}^{n-1}a_{k,1}^k\Big)e_n,\\
a_{2,2}^k &=0, \qquad 4\leq k\leq n.
\end{align*}

The equations $(d^2\varphi)(e_1,x_1,e_1)=0$ and $(d^2\varphi)(e_2,x_1,e_1)=0$ imply
\[a_{1,n+1}^i=-a_{n+1,1}^i, \qquad i=1,n+1,n+2.\]

From the conditions $(d^2\varphi)(e_i,x_1,e_j)=0$ for $2\leq i,j\leq n$, we conclude
\begin{align*}
\varphi(x_1,e_2)  & =-a_{2,2}^3e_1, \\
\varphi(x_1,e_3)  & =(a_{1,2}^1+a_{2,1}^1)e_1,\\
\varphi(x_1,e_i) &  =(a_{i-1,1}^1-a_{i-2,1}^{n+1})e_1,\qquad 4\leq i\leq n.\\
\end{align*}
The equalities $(d^2\varphi)(e_i,e_j,x_2)=0, \ 1\leq i,j\leq n$, determine
\begin{align*}
\varphi(e_1,x_2) & =a_{1,n+2}^1e_1+\sum\limits_{k=2}^{n-1}a_{1,1}^{k+1}e_k+a_{1,n+2}^ne_n,\\
\varphi(e_2,x_2) & =-a_{2,2}^3e_1+\sum\limits_{k=2}^na_{2,n+2}^ke_k-a_{1,2}^1x_1+(a_{1,2}^1-a_{2,2}^2)x_2,\\
\varphi(e_3,x_2)& =(a_{1,2}^1+a_{2,1}^1)e_1+(a_{2,n+2}^2+a_{1,n+2}^1)e_3+\sum\limits_{k=4}^na_{2,n+2}^{k-1}e_k+a_{2,1}^{n+1}x_1+a_{2,1}^{n+2}x_2,\\
\varphi(e_i,x_2) & =(a_{i-1,1}^1-a_{i-2,1}^{n+1})e_1+(a_{2,n+2}^2+(i-2)a_{1,n+2}^1)e_i+
\sum\limits_{k=i+1}^na_{2,n+2}^{k-i+2}e_k+a_{i-1,1}^{n+1}x_1+a_{i-1,1}^{n+2}x_2,\ 4\leq i\leq n.
\end{align*}

From the equality $(d^2\varphi)(e_2,x_2,e_1)=0$ we get $a_{1,n+2}^1=-a_{n+2,1}^1$.

Given the restrictions above, by using the property of cocycle for $(d^2\varphi)(x_i,x_j,e_1)=0$ for $1\leq i,j\leq 2$, we obtain
\begin{align*}
\varphi(x_1,e_1) & =a_{n+1,1}^1e_1+a_{1,1}^3e_2+\sum\limits_{k=3}^{n}a_{n+1,1}^ke_k+a_{n+1,1}^{n+1}x_1+a_{n+1,1}^{n+2}x_2,\\
\varphi(x_1,x_1) & =\sum\limits_{k=2}^{n-2}(k-1)(a_{n+1,1}^{k+1}-a_{1,1}^{k+2})e_k+((n-2)a_{n+1,1}^n-a_{1,n+1}^n)e_{n-1}+a_{n+1,n+1}^ne_n,\\
\varphi(x_2,x_1) & =a_{n+2,n+1}^1e_1+\sum\limits_{k=2}^{n-1}(k-1)a_{n+2,1}^{k+1}e_k+a_{n+2,n+1}^ne_n+a_{n+2,n+1}^{n+2}x_2,\\
\varphi(x_1,x_2) & =a_{n+1,n+2}^1e_1+\sum\limits_{k=2}^{n-2}(a_{n+1,1}^{k+1}-a_{1,1}^{k+2})e_k+(a_{n+1,1}^n-a_{1,n+2}^n)e_{n-1}+a_{n+1,n+2}^ne_n+a_{n+1,n+2}^{n+2}x_2,\\
\varphi(x_2,x_2) & =\sum\limits_{k=2}^{n-1}a_{n+2,1}^{k+1}e_k+a_{n+2,n+2}^ne_n.
\end{align*}

By checking the general condition of cocycle for the other basis elements, we get the following
\begin{align*}
(d^2\varphi)(e_1,x_1,x_2)=0  & \Rightarrow \ a_{1,n+1}^n=(n-2)a_{1,n+2}^n;\\
(d^2\varphi)(e_2,x_1,x_2)=0  & \Rightarrow \ a_{n+1,n+2}^1=-a_{2,n+2}^3,\ a_{n+1,n+2}^{n+2}=a_{2,n+2}^k=0, \quad 4\leq k\leq n;\\
(d^2\varphi)(e_2,x_2,x_1)=0  &  \Rightarrow \ a_{n+2,n+1}^1=a_{2,n+2}^3,\ a_{n+2,n+1}^{n+2}=0.
\end{align*}

Considering the equations $(d^2\varphi)(x_i,x_j,x_k)=0$ for $1\leq i,j,k\leq 2$, we have
\[x_i\varphi(x_j,x_k)-\varphi(x_i,x_j)x_k+\varphi(x_i,x_k)x_j=0,\]
from which we obtain
\[a_{n+1,n+1}^n=(n-1)a_{n+1,n+2}^n, \qquad a_{n+2,n+1}^n=(n-1)a_{n+2,n+2}^n.\]

Thus, we have a general form of  the 2-cocycle $\varphi$.
\end{proof}

Proposition \ref{pr2} implies the following corollary.

\begin{cor}
$\dim ZL^2(R(F_n^1),R(F_n^1))=(n+2)^2-3$ and $\dim HL^2(R(F_n^1),R(F_n^1))=0$.
\end{cor}

Thus, according to the results of the paper \cite{Bal}, we have the following  theorem.
\begin{teo}
The algebra $R(F_n^1)$ is rigid.
\end{teo}

\section*{Acknowledgments}
This work was partially supported by Ministerio de Econom\'ia y Competitividad (Spain),
grant MTM2013-43687-P (European FEDER support included),   by Xunta de Galicia, grant GRC2013-045 (European FEDER support included), Grant
No. 0828/GF4 of Ministry of Education and Science of the Republic of Kazakhstan. The last named
author was partially supported by a grant from the Simons Foundation.


\end{document}